\documentclass{amsproc}
\usepackage{amsmath}
\usepackage{latexsym,amsfonts,amssymb,mathrsfs, mathtools,color}
\usepackage [all,2cell]{xy}
\usepackage{hyperref}
\usepackage{mathrsfs}
\UseAllTwocells

\newtheorem{thm}{Theorem}[section]

\newtheorem{prop}{Proposition}[section]
   \makeatletter\let\c@prop\c@thm\makeatother
   
\newtheorem*{prop*}{Proposition}

\newtheorem{lem}{Lemma}[section]
   \makeatletter\let\c@lem\c@thm\makeatother

\newtheorem*{thmA}{Theorem A}
\newtheorem*{thmB}{Theorem B}

\theoremstyle{definition}
\newtheorem{defn}{Definition}[section]
   \makeatletter\let\c@defn\c@thm\makeatother

\theoremstyle{remark}

\newtheorem*{rmk}{Remark}

\newcommand{\cA}{\mathcal{A}}
\newcommand{\cB}{\mathcal{B}}
\newcommand{\cC}{\mathcal{C}}
\newcommand{\cD}{\mathcal{D}}
\newcommand{\cP}{\mathcal{P}}
\newcommand{\cS}{\mathcal{S}}
\newcommand{\cT}{\mathcal{T}}
\newcommand{\cW}{\mathcal{W}}
\newcommand{\cat}{\cC\!\mathit{at}}
\renewcommand{\top}{\cT\!\mathit{op}}
\newcommand{\set}{\cS\!\mathit{et}}

\newcommand{\sset}{\mathit{s}\set}
\newcommand{\gcat}{G\cat}
\newcommand{\gtop}{G\top}
\newcommand{\gset}{G\set}

\newcommand{\gsset}{G\sset}
\newcommand{\gc}{G\cC}
\newcommand{\gd}{G\cD}
\newcommand{\og}{\mathscr{O}_G}
\newcommand{\ocat}{\og\textrm{-}\cat}
\newcommand{\otop}{\og\textrm{-}\top}
\newcommand{\osset}{\og\textrm{-}\sset}
\newcommand{\oc}{\og\textrm{-}\cC}
\newcommand{\od}{\og\textrm{-}\cD}

\newcommand{\bg}{BG}
\newcommand{\ob}{\text{ob}}

\DeclareMathOperator{\colim}{colim}
\DeclareMathOperator{\id}{id}

\begin{document}
\title{A model structure on $\gcat$}

\author[A.M. Bohmann]{Anna Marie Bohmann}
\address{Department of Mathematics, Northwestern University, Evanston, IL 60208, USA}
\email{bohmann@math.northwestern.edu}
\author[K. Mazur]{Kristen Mazur}
\address{Department of Mathematics, Lafayette College, Easton, PA 18042, USA}
\email{mazurk@lafayette.edu}
\author[A. Osorno]{Ang\'{e}lica M. Osorno}
\address{Department of Mathematics, Reed College, Portland, OR 97202, USA}
\email{aosorno@reed.edu}
\author[V. Ozornova]{Viktoriya Ozornova}
\address{Fachbereich Mathematik, Universit\"{a}t Bremen, 28359 Bremen, Germany}
\email{ozornova@math.uni-bremen.de} 
\author[K. Ponto]{Kate Ponto}
\address{Department of Mathematics, University of Kentucky, Lexington, KY 40506, USA}
\email{kate.ponto@uky.edu}
\author[C. Yarnall]{Carolyn Yarnall}
\address{ Department of Mathematics, Wabash College, Crawfordsville, IN 47933, USA}
\email{yarnallc@wabash.edu}

\begin{abstract}
We define a model structure on the category $\gcat$ of small categories with an action by a discrete group $G$  by 
lifting the Thomason model structure on $\cat$.  We show 
 there is a Quillen equivalence between $\gcat$ with this model structure and $\gtop$ with the standard model structure.
\end{abstract}

\maketitle

\section*{Introduction}

There are familiar adjunctions 
\[\xymatrix{
\cat \ar[r]<1ex>^{N}& \sset  \ar[l]<1ex>^{c} \ar[r]<-1ex>_{|-|} &  \top \ar[l]<-1ex>_{S_{\bullet}(-)} 
}\]
between the categories of categories, simplicial sets, and topological spaces, and for the standard model structure on 
$\sset$ and the Quillen model structure on $\top$ 
the adjunction on the right is a Quillen equivalence.  
In \cite{thomason} Thomason defined a model structure on $\cat$ and showed that the adjunction 
\[\xymatrix{
\cat \ar[r]<1ex>^{Ex^2N}& \sset  \ar[l]<1ex>^{cSd^2} 
}\]
is a Quillen equivalence.
In Thomason's model structure a functor 
$F\colon \cA\to \cB$  is a weak equivalence if $Ex^2N(F)$ is a weak equivalence in $\sset$ or, equivalently, $BF$ is a weak equivalence of 
topological spaces.
A functor $F$  is a fibration if $Ex^2N(F)$ is a  fibration in $\sset$.  As shown in \cite{FP}, this model structure  is cofibrantly generated.

In this paper we use results by Stephan \cite{stephanthesis} to extend Thomason's model structure to the category of categories with  an action by a discrete group $G$.
We let $\bg$ be the category with one object and endomorphisms given by the group $G$ and 
define the \textbf{category of $G$ objects} in a category $\cC$, denoted by $\gc$,  to be 
the category of functors \[\bg\to \cC\] and natural transformations. 

\begin{rmk}
Explicitly, an object of $\gc$ is an object $C$ of $\cC$ along with \emph{isomorphisms} $\sigma_g\colon C\to C$
so that $\sigma_g\sigma_h=\sigma_{gh}$ and $\sigma_e=\id$.
As an alternative, we could consider $G$ objects where the morphisms $\sigma_g$ are equivalences and the group identities hold up to 
natural isomorphism.  This corresponds to a pseudofunctor from $\bg$ to the 2-category of categories and would define group actions up to homotopy after passing to topological spaces. 
Since our primary interest is in the comparison between $\gcat$ and $\gtop$ we will only consider strict actions. 
\end{rmk}

If $\cC$ is a model category we can define a model structure on $\gc$ where the fibrations and weak equivalences 
are maps that are fibrations or weak equivalences in $\cC$.
Unfortunately, this perspective does not capture the desired homotopy theory.  
This is perhaps most familiar in the case of $\gtop$, where the desired notion of $G$-weak equivalence 
is a map that induces a non-equivariant weak equivalence on fixed point spaces for all subgroups of $G$. 

Given a subgroup $H$ of $G$, we have a functor $(-)^H\colon \gc\to \cC$ defined by $X^H= \lim_{BH}X$.  This notion coincides with the usual definition of the 
fixed point functor in the case that $\cC$ is any of $\set$, $\top$, $\sset$ or $\cat$. In the case of $\cat$, $C^H$ is the subcategory of $C$ consisting of those objects and morphisms fixed by all $h\in H$.
Let $\og$ be the orbit category of $G$; it has objects the orbits $G/H$ for all subgroups $H$ and morphisms all equivariant maps. Then
an object $X\in \gc$ defines a functor 
\[\Phi(X)\colon\og^{op}\to \cC\]
by $\Phi(X)(G/H)=X ^H$.
 If we let 
$\oc$ be the category of functors  \[\og^{op}\to \cC\]
  we can define a
functor $\Phi\colon \gc \to  \oc$  as above.
 The functor $\Phi$ has a left adjoint \[\Lambda\colon \oc \to \gc,\] defined by $\Lambda (Y)=Y(G/e)$, where the $G$-action is inherited from the automorphisms of the object $G/e$ in $\og$.

If $\cC$ is a cofibrantly generated model category, such as $\top$, $\sset$ or Thomason's model structure on $\cat$, 
there is a model structure on $\oc$ where the fibrations and weak equivalences 
are defined levelwise. This is 
the \textbf{projective model structure} on the category $\oc$. For the category of topological spaces, or simplicial sets,  this model structure captures the desired equivariant homotopy type.

For some categories $\cC$ we can use the functor $\Phi$ to lift the projective model structure from $\oc$ to $\gc$. Then a map in $\gc$ is a 
fibration or weak equivalence if it is one after applying $\Phi$.  In the case of topological spaces this is the 
usual model structure on $\gtop$ \cite[III.1.8]{MM}. In \cite{elmendorf}, 
Elmendorf constructed a functor $\otop \to \gtop$ that was an inverse of $\Phi$ up to homotopy, thus showing that the homotopy categories of $\gtop$ and $\otop$ were equivalent. Later Piacenza \cite{piacenza} showed that the adjunction given by $\Phi$ and $\Lambda$ is a Quillen equivalence if $\gtop$ has this model structure and $\otop$ has the 
projective model structure. Note that Elmendorf's functor can be thought of as the composition of the cofibrant replacement in $\otop$ followed by $\Lambda$.

In this paper we prove a similar result for $\cat$.

\begin{thmA}\label{thm:main}  If $G$ is a  discrete group there is a model structure on $\gcat$ where a functor is a fibration or weak equivalence if it is so after applying $\Phi$. Using this model structure 
  the $\Lambda\text{--}\Phi$ adjunction is  a Quillen 
  equivalence between $\gcat$ and $\ocat$. 
\end{thmA}

More can be said about this model structure.  
Since $\oc$ and $\gc$ are both diagram categories, an adjunction $L\colon \cC\rightleftarrows \cD\colon R$ defines adjunctions 
\[L_*\colon \oc\rightleftarrows \od\colon R_*\text{ and } L_*\colon \gc\rightleftarrows \gd\colon R_*\]
and so the classical adjunctions relating $\cat$, $\sset$, and $\top$ define adjunctions 
\[\xymatrixrowsep{2cm} 
\xymatrixcolsep{2cm} \xymatrix{
\gcat \ar[r]<1ex>^{Ex^2N} \ar[d]<1ex>^{\Phi} & \gsset \ar[d]<1ex>^{\Phi } \ar[l]<1ex>^{cSd^2} \ar[r]<-1ex>_{|-|} &  \gtop \ar[d]<1ex>^{\Phi} \ar[l]<-1ex>_{S_{\bullet}(-)} \\
\ocat \ar[r]<1ex>^{Ex^2N} \ar[u]<1ex>^{\Lambda} & \osset  \ar[l]<1ex>^{cSd^2}\ar[r]<-1ex>_{|-|} \ar[u]<1ex>^{\Lambda} & \otop.  \ar[l]<-1ex>_{S_{\bullet}(-)} \ar[u]<1ex>^{\Lambda}
} \]
The usual Quillen equivalences between $\cat$, $\sset$ and $\top$ are known to  induce Quillen equivalences between $\ocat$, $\osset$ and $\otop$.

\begin{thmB}
The adjunctions in the top row of the diagram above are Quillen equivalences. 
\end{thmB}

\subsection*{Acknowledgments}
We would like to thank the organizers of the Women in Topology Workshop for providing a wonderful opportunity for collaborative research. Thanks to Peter May for suggesting this question, to Emily Riehl for generously sharing her category theory insights, and to Marc Stephan for suggesting further generalizations. Many thanks to Teena Gerhardt for her help with the early stages of this project. The fourth author would like to thank SFB 647 Space--Time--Matter for supporting her travel to the workshop. The fifth author was partially supported by NSF grant DMS-1207670.  Finally we would like to thank the Banff International Research Station and the Clay Mathematics Institute for supporting the workshop.

\section{Model structures on $G$-categories}\label{sec:contex}

Let $\cC $ be a cofibrantly generated model category. To lift the model structure from $\cC$ to the category $G\cC$ we need some compatibility between 
the model structure on $\cC$ and the group action.  The relevant notion of compatibility is captured using the 
fixed point functors.

\begin{defn}\label{defn:cellular}
A fixed point functor  $(-)^H\colon \gc \to \cC$ is \textbf{cellular} if 
\begin{enumerate}
  \item it preserves directed colimits of diagrams where each arrow is a non-equivariant cofibration after applying the forgetful functor $\gc\to \cC$,
  \item it preserves pushouts of diagrams where one leg is given by 
  \[ G/K\otimes f\colon G/K\otimes A\to G/K\otimes B\]
  for some closed subgroup $K$ of $G$ and a cofibration $f\colon A\to B$ in $\cC$, and 
  \item for any closed subgroup $K$ of $G$ and any object $A$ of $\cC$ the induced map \[(G/K)^H\otimes A\to (G/K\otimes A)^H\] is an isomorphism in $\cC.$
\end{enumerate}
\end{defn}
Note that since $\cC$ is cocomplete, for a $G$-set $X$ and an object $A$ of $\cC$ we have the categorical tensor $X\otimes A$ which is the $G$-object $\coprod_X A$ with $G$-action induced by the $G$-action on $X$.

In \cite{stephanthesis}, Stephan gives conditions to lift a model structure from $\oc$ to $\gc$.
\begin{thm} \cite[Proposition 2.6, Lemma 2.9]{stephanthesis} \label{thm:stephanlifting}
  Let $G$ be a discrete group, 
   $\cC$ be a model category which is cofibrantly generated and assume
    for any subgroup $H\leq G$ the $H$-fixed point functor $(-)^H\colon \gc\to \cC$ is cellular.  Then there is a \textbf{fixed point model structure} on $\gc$ where a map $f$ in $\gc$ is a fibration or weak equivalence if and only if $\Phi(f)$ is a  fibration or weak
  equivalence in the projective model structure on $\oc$.
  Additionally, there is a Quillen equivalence
  \[ \Lambda\colon \oc\rightleftarrows \gc:\!\Phi\]
  between $\oc$ with the projective model structure and $\gc$ with this model structure. 
\end{thm}

 Reflecting the hypothesis of this theorem, for the rest of this section we assume that $G$ is a discrete group.

This theorem can be made functorial with respect to Quillen adjunctions.  

\begin{thm}\label{thm:expandedmain}
  Let $\cC$ and $\cD$ be cofibrantly generated model categories 
  satisfying the hypotheses of \autoref{thm:stephanlifting}.
  If \[L\colon \cC \rightleftarrows \cD :\!R\]
  is a Quillen adjunction (resp. Quillen equivalence)  then there is an induced Quillen adjunction (resp. Quillen equivalence)
  \[L_*\colon \gc \rightleftarrows \gd :\! R_*\]
  where $\gc$ and $\gd$ have fixed point model structures.  
\end{thm}

\begin{proof}
  To show we have a Quillen adjunction 
  it is enough to show that $R_*\colon \gd\to \gc$ is a right Quillen functor, that is,   
  to show that $R_*$ preserves fibrations and acyclic fibrations. We will show $R_*$ preserves fibrations; the case for acyclic fibrations is similar.

  Let $f\colon X\to Y$ be a fibration in $\gd$. Since $\gd$ has the fixed point model structure, 
  fibrations are created in $\od$.  
  Thus $\Phi f\colon \Phi X\to \Phi Y$ is also a fibration in $\od$. 

 By assumption, $R\colon \cD \to\cC$ is right Quillen and thus by \cite[Theorem 11.6.5]{hirschhorn}, $R_*\colon \od\to \oc$ is also right Quillen. Thus $R_*\Phi f\colon R_*\Phi X\to R_*\Phi Y$ is a fibration.

As a right adjoint, $R$ commutes with limits.  This allows us to equate $\Phi R_*$ and $R_*\Phi$.  To be explicit, consider any $H\leq G$ and $X\colon \bg\to \cD$. By definition, $(\Phi R_*X)(G/H)=(R_*X)^H$ is given by 
$\lim_{BH} RX,$
the limit along $BH$ of the composite functor $RX\colon \bg\to \cD\to \cC$. Since $R$ commutes with limits, we obtain the identification
\[\lim_{BH} RX=R\lim_{BH} X.\]
This later object is the definition of $(R_* \Phi X)(G/H)$, and therefore the maps $R_*\Phi f$ and $\Phi R_*f$ are equal. This means that $\Phi R_* X\xrightarrow{\Phi R_* f}\Phi R_*Y$ 
  is a fibration, and thus, since fibrations  in $\gc$ are created under 
  $\Phi$, $R_*X\xrightarrow{R_*f} R_*Y$ is a fibration.

  Suppose $L\colon \cC\rightleftarrows\cD:\! R$ is a Quillen equivalence. To show the adjunction 
  $\gc \rightleftarrows \gd$ is a Quillen equivalence, we apply the 2-out-of-3 property for 
  Quillen equivalences  \cite[Corollary 1.3.15]{hovey}.  We  then have a diagram of Quillen adjunctions, in which both the diagrams of the left adjoints and the right adjoints commute,
  \[
   \xymatrix{\gc\ar@<.5ex>[r]^{L_*}\ar@<.5ex>[d]^{\Phi} &\gd\ar@<.5ex>[l]^{R_*}\ar@<.5ex>[d]^{\Phi}\\
   \oc\ar@<.5ex>[u]^{\Lambda}\ar@<.5ex>[r]^{L_*}& \od\ar@<.5ex>[u]^{\Lambda}\ar@<.5ex>[l]^{R_*}}
  \]
  such that bottom and two side adjunctions are Quillen equivalences.  Thus the top adjunction must be a Quillen equivalence as well.
\end{proof}

After we verify that $\cC at$ satisfies the conditions of \autoref{thm:stephanlifting} in the next section, \autoref{thm:expandedmain} completes the proof of Theorem B.

We now record that Stephan's construction preserves right properness. 
	
\begin{prop}\label{prop:leftproper}
  Let $\cC$ be a cofibrantly generated model category that is right proper and satisfies 
  the conditions of \autoref{thm:stephanlifting}.  
  Then the fixed point model structure on $\gc$ is right proper.
\end{prop}

\begin{proof}
 Suppose $\cC$ is right proper and consider a pullback diagram in $\gc$
\[ \xymatrix{X \ar[r]^{f'}\ar[d]& Y\ar[d]^{h}\\
Z\ar[r]_{f} & W}
\]
 where $h$ is a fibration and $f$ is a weak equivalence. We must show that $f'$ is also a weak equivalence. Since weak equivalences and fibrations in $\gc$ are created by the functor $\Phi\colon \gc\to \oc$ and $\oc$ is right proper
  \cite[Thm.~13.1.14]{hirschhorn}
  this follows from the fact that $\Phi$ is a right adjoint and thus commutes with pullbacks.
\end{proof}

To apply \autoref{thm:stephanlifting} to  the category $\cat$  we will show this category and its fixed point functors 
satisfy conditions that imply the fixed point functors are cellular.

\begin{prop}\label{prop:cellularnew}
Let  $H$ be a subgroup of $G$ and
   $\cC$ be a cofibrantly generated model category.   Assume the $H$-fixed point functor $(-)^H\colon \gc\to \cC$ 
  \begin{enumerate}
   \item  
  preserves filtered colimits in $\gc$ where each arrow is a non-equivariant cofibration after applying the forgetful functor $\gc\to \cC$,
	\item preserves pushouts of diagrams where one leg is given by 
	  \[ G/K\otimes f\colon G/K\otimes A\to G/K\otimes B\]
	  for a subgroup $K$ of $G$ and a generating cofibration $f\colon A\to B$ in $\cC$, and 
	  \item for any subgroup $K$ of $G$ and any object $A$ of $\cC$ the induced map \[(G/K)^H\otimes A\to (G/K\otimes A)^H\] is an isomorphism in $\cC.$
\end{enumerate}
Then the $H$-fixed point functor is cellular.
\end{prop}

We postpone the proof to \S\ref{sec:gen}, but first observe that it allows us to prove a dual result to \autoref{prop:leftproper}.  
This proof is also postponed to \S\ref{sec:gen}.

\begin{prop}\label{prop:left}
  Let $\cC$ be a cofibrantly generated model category that is left proper and satisfies  the conditions of \autoref{prop:cellularnew}. Then the fixed point model structure on $\gc$ is left proper.
\end{prop}

\section{The model category $\gcat$}
\label{sec: DirectedColimitCondition}

In this section we will show that $\cat$ satisfies the hypotheses of \autoref{prop:cellularnew} proving Theorem A.
We start by 
recalling an explicit description of the generating cofibrations and generating acyclic cofibrations in Thomason's model structure on $\cat$.

\begin{thm}\cite[Thm. 6.3]{FP}\label{thm:gencof}
   The Thomason model structure on  $\cat$ is cofibrantly generated with generating cofibrations 
  \[ \{c\mathrm{Sd}^2\partial \Delta[m]\to c\mathrm{Sd}^2\Delta[m] \mid m\geq 0 \}\]
  and generating acyclic cofibrations
  \[ \{ c\mathrm{Sd}^2\Lambda^k[m]\to c\mathrm{Sd}^2\partial \Delta[m] \mid m\geq 1 \text{ and } 0\leq k \leq m \}.\]
  Here $c$ is the left adjoint of the nerve functor and $\mathrm{Sd}$ is barycentric subdivision.
\end{thm}

It is important to  note that the sources and targets of the generating cofibrations and acyclic cofibrations are posets. 
To verify the conditions of  \autoref{prop:cellularnew} we will consider a more general 
collection of maps, the {\it Dwyer maps of posets}, rather than working directly with these generating cofibrations and acyclic cofibrations.

Recall that  subcategory $\cA$ of a category $\cB$ is a \textbf{sieve} if for every morphism $\beta\colon b\to a$ in $\cB$ with target $a$ in $\cA$, 
both the object $b$ and the morphism $\beta$ lie in $\cA$. A \textbf{cosieve} is defined dually.
\begin{defn}\cite{thomason}
   A sieve inclusion $\cA \to \cB$ is a \textbf{Dwyer map} if  there is a cosieve $\cW$ in $\cB$ containing $\cA$ 
  so that the inclusion functor $i\colon \cA \to \cW$ admits a right adjoint $r\colon \cW \to \cA$ satisfying 
  $ri=\id_{\cA}$ and  the unit of this adjunction is  the identity. 
\end{defn}

The generating cofibrations and acyclic cofibrations are Dwyer maps (of posets), but as observed in \cite{cisinski}, Dwyer maps are not closed under retracts and there 
are cofibrations that are not Dwyer maps. 

We next show $\mathcal{C}at$ satisfies the conditions of \autoref{prop:cellularnew}.  Condition 3 of \autoref{prop:cellularnew} 
is satisfied because the action of $G$ on $G/K\otimes A$ is entirely through the action of $G$ on $G/K$.  
 Condition 1 will follow from the next proposition, since all cofibrations in $\cat$ are monomorphisms.

\begin{prop}
Let $H$ be a subgroup of  $G$. Then the fixed point functor \[(-)^H\colon \gcat \to \cat\] preserves filtered colimits for which each morphism is an underlying monomorphism in $\cat$.
\end{prop}

\begin{proof}
  Let $I$ be a filtered category and $F$ be a functor from $I$ to $\gcat$ such that for each morphism $i$ in $I$, $F(i)$ is a monomorphism.    First note that $N \colim_{I}(F(i)^H) \cong \colim_{I}\left(N\left(F(i)^H\right)\right)$ since the nerve commutes with 
  filtered colimits \cite{lack}. 
   The nerve is a right adjoint and taking fixed points is a limit, so we have an isomorphism
  $N\left(F(i)^H\right)\cong \left(NF(i)\right)^H$.  Together these give an isomorphism
  \[N \colim_{I}(F(i)^H) \cong \colim_{I}\left(\left(NF(i)\right)^H\right).\]
   Taking $H$-fixed points preserves filtered colimits  in $\gset$ where each arrow is a monomorphism. This extends to $\sset$ since fixed points and colimits in $\sset$  are computed levelwise. Since each $F(i)$ is a monomorphism, so is $NF(i)$, thus we have  
  \begin{equation}\label{eq:iso}\colim_{I}\left(\left(NF(i)\right)^H\right) \cong  \left(\colim_I NF(i)\right)^H\cong \left(N\colim_IF(i) \right)^H.\end{equation}
  Finally, we have an isomorphism $\left(N\colim_IF(i) \right)^H\cong N\big(\colim_I F(i)\big)^H$ since the nerve is a right adjoint.
  Together this gives an isomorphism \[N \colim_{I}(F(i)^H) \cong N\big(\colim_I F(i)\big)^H.\]
The nerve is fully faithful, so it follows that $\colim_{I}(F(i)^H) \cong \big(\colim_I F(i)\big)^H$, thus completing the proof.
\end{proof}

\begin{rmk} Note that if $G$ is finite, then $(-)^H$ preserves all filtered colimits.  In this case \eqref{eq:iso} follows from 
the fact that 
finite limits and filtered colimits commute in $\set$ and this extends to $\sset$ since  limits and colimits in $\sset$  are computed levelwise.  
\end{rmk}

We will verify the second condition in \autoref{prop:cellularnew}
for Dwyer maps of posets 
since they allow simple descriptions of pushouts of categories.

\begin{prop}\label{pushout}
Let $\cA \to \cB$ be a Dwyer map of posets and suppose the diagram 
\begin{eqnarray*}
\xymatrix{
 G/K\times \cA \ar[r]\ar[d]_{F} & G/K\times \cB \ar[d]\\
  \cC \ar[r] & \cD
}
\end{eqnarray*}
is a pushout diagram in $\gcat$. Then this diagram remains a pushout after taking $H$-fixed points.
\end{prop}

The proof of this proposition is based on a very explicit description of the morphisms in $\cD$.  We give that description first and then continue to 
the proof of the proposition.

\begin{lem}\label{lem:pushoutcatsieve}
  Let $i\colon \cA \to \cB$ be a Dwyer map between posets with cosieve $\cW$ and retraction $r$, and let $F\colon \cA \to \cC$ be any functor. 
  If $\cD$ is the pushout of $i$ and $F$, the set of objects of $\cD$ can be identified with \[\ob(\cC) \amalg \left ( \ob(\cB)\setminus \ob(\cA)\right).\]
  If $c$ is an object of $\cC$ and $b$ is an object of $\cB$ that is not an object of $\cA$, then 
  \[\cD(c,b)\cong\cC(c, F(r(b))),\] if $b$ is in $\cW$, and is otherwise empty.
\end{lem}

\begin{proof}
  The proof of \cite[Proposition 5.2]{FL} gives a simple description for the pushout $\cD$ of a  full inclusion $i\colon \cA\to \cB$ and a
  functor $F\colon \cA \to \cC$.  In  the case when $i$ is a sieve, the description is as follows. The objects of $\cD$ are $\ob(\cC) \amalg \left ( \ob(\cB)\setminus \ob(\cA)\right)$
  and  some of the morphisms are given by 
  \[\cD(d,d')=\left\{\begin{array}{ll}\cB(d,d')&\text{if }d,d'\in \ob (\cB) \setminus \ob (\cA),\\
    \cC(d,d')&\text{if }d,d'\in \ob (\cC),\\
    \emptyset&\text{if }d\in \ob (\cB) \setminus \ob (\cA)\text{ and } d'\in \ob (\cC).
    \end{array}\right.
  \]
  For an object $c$ of $\cC$ and an object $b$ of $\cB$ not in $\cA$, the morphisms from $c$ to $b$ in $\cD$ are equivalence classes of pairs $(\beta, \gamma)$ where
  $\beta$ is a morphism $a\to b$ in $\cB$  for some $a\in \cA$ and $\gamma$ is a morphism $c\to F(a)$.
  The equivalence relation on these pairs is generated by $(\beta\alpha, \gamma)\sim (\beta, F(\alpha)\gamma)$ for 
  $\alpha$ in $\cA$, whenever the compositions in question are defined. The equivalence relation is compatible with composition.

  Now assume that $\cA \to \cB$ is a Dwyer map between posets. We denote the counit of the adjunction between the inclusion $\cA \to \cW$ 
  and the retraction $r$ by  $\varepsilon$. If $(\beta, \gamma)$ is a pair of morphisms as above, then $\beta$ is in $\cW$ by the definition of 
  cosieve, and \[\beta=\varepsilon_b r(\beta)\] since the source of $\beta$ is in $\cA$.  Since $r(\beta)\in \cA$, $(\beta,\gamma)$ is 
  equivalent to $(\varepsilon_b, F(r(\beta)) \gamma)$ and, as the reader can check,   every equivalence class has a unique representative of the form $(\varepsilon _b, \gamma)$.
\end{proof}

\begin{proof}[Proof of \autoref{pushout}]
  We must show that if the diagram on the left is a pushout and 
  $\cA \to \cB$ is a Dwyer map of posets then the diagram on the right is also a pushout. 
  \[\xymatrix{
  G/K\times\cA \ar[r]\ar[d]_{F} & G/K\times \cB \ar[d]& (G/K)^H\times \cA \ar[r]\ar[d]_{F^H} & (G/K)^H \times \cB \ar[d]\\
  \cC \ar[r] & \cD &  \cC^H \ar[r] & \cD^H
  }\]

 First observe that since $\gcat$ is a diagram category, the pushout is computed in the underlying category $\cat$.  
Since $G/K$ and $(G/K)^H$ are sets considered as discrete categories, the top horizontal maps on both 
diagrams are also Dwyer maps of posets. Thus we can apply \autoref{lem:pushoutcatsieve} to both diagrams.
  
The objects of $\mathcal{D}$ are
  \[
   \ob(\cC) \amalg \bigl( G/K \times \ob(\cB)\setminus \ob(\cA) \bigr)
  \]
  so the objects of $\cD^H$ are given by
  $
   \ob(\cC)^H \amalg \bigl( (G/K)^H \times (\ob(\cB)\setminus \ob(\cA)) \bigr) .
  $
  The objects of the pushout $\cP$ of $(G/K)^H\times \cA\to (G/K)^H\times \cB$ and $F^H$ are identical to those in $\cD^H$ 
  and the induced map from this pushout to $\cD^H$ is an isomorphism on objects.

Applying \autoref{lem:pushoutcatsieve}, the morphisms of $\cD$ are 
  \[\cD(d,d')\!=\!\left\{\begin{array}{ll} \{\id_{gK}\}\!\times\!\cB(b,b')&\begin{array}{l}\text{if }d=(gK,b),d'=(gK,b') \text{ for } \\\,\,b,b'\in \ob (\cB)\setminus \ob(\cA),\end{array}\vspace{.05in}\\\vspace{.05in}
    \cC(d,d')&\begin{array}{l}\text{if }d,d'\in \ob (\cC),\end{array}\\\vspace{.05in}
   \{\id_{gK}\} \!\times\! \cC(d, F(gK, r(b)))&\begin{array}{l}\text{if }d\in \ob (\cC)\text{ and } d'=(gK, b) \text{ for } \\\,\,b\in \ob (\cB)\setminus \ob(\cA),\end{array}\\\vspace{.05in}
  \emptyset& \begin{array}{l}\text{if }d=(gK, b)\text{ and } d'\in \ob (\cC) \text{ for } \\\,\,b\in \ob (\cB)\setminus \ob(\cA).\end{array}
    \end{array}\right.
  \]
  and so for objects $d$ and $d'$ in $\cD ^H$  we have 
  \[\cD^H\!(d,d')\!=\!\left\{\begin{array}{l@{\hspace{2pt}}l} \{\id_{gK}\}\!\times\!\cB(b,b')&\begin{array}{l}\text{if }d=(gK,b),d'=(gK,b') \text{ for } \\ gK\in (G/K)^H  
	\text{,  } b,b'\in \ob (\cB)\!\setminus\! \ob(\cA),\end{array}\vspace{.05in}\\
    \cC^H(d,d')&\begin{array}{l}\text{if }d,d'\in \ob (\cC^H),\end{array}\vspace{.05in}\\
    \{\id_{gK}\} \!\times\! \cC^H(d, F( gK,r(b)))&\begin{array}{l}\text{if }d\in \ob (\cC^H)\text{ and } d'=(gK, b) \text{ for } \\ 
	gK\in (G/K)^H   \text{,  }b\in \ob (\cB)\setminus \ob(\cA).\end{array}\vspace{.05in}\\
  \emptyset& \begin{array}{l}\text{if }d=(gK, b)\text{ and } d'\in \ob (\cC^H) \text{ for } \\gK\in(G/K)^H   \text{,  }b\in \ob (\cB)\setminus \ob(\cA).\end{array}
    \end{array}\right.
  \]

  For the pushout $\cP$, the analogous statement holds, and thus we have the same description for the morphism sets of $\cP$ and $\cD^H$ and the induced map 
  $\cP\to \cD^H$ is an isomorphism on morphism sets.  
\end{proof}

\section{Cellular functors and left proper model structures}\label{sec:gen}
We now return to the proof of \autoref{prop:cellularnew}.  We only need to show that condition (2) in \autoref{prop:cellularnew} 
can be extended from generating cofibrations to all cofibrations. This is a direct consequence of the following lemma, 
using the fact that cofibrations are retracts of transfinite compositions of pushouts of generating cofibrations.

\begin{lem}\label{gencof}
 Let $F\colon\mathcal{C} \to \mathcal{D}$ be a functor between cocomplete categories and $I$ a set of morphisms of $\mathcal{C}$. Let $J$ be the collection of all retracts of transfinite compositions of pushouts of morphisms in $I$. If $F$ preserves
filtered colimits along morphisms in $J$ and pushouts along morphisms in $I$ then $F$ preserves pushouts along morphisms in $J$.
\end{lem}

\begin{proof}
  Suppose we have a diagram 
  \[\xymatrix{
   A \ar[r]\ar[d]^{i} & C\ar[r]\ar[d]^{i'} & C'\ar[d]\\
   B\ar[r] & D\ar[r] & D'
  }\]
  where both small squares are pushouts and $i$ is in $I$.  Then the exterior is a pushout.
  Applying $F$ we see that both the left square and the outside rectangle remain pushouts.  This 
  implies the right square is a pushout, so $F$ preserves pushouts along pushouts of morphisms in $I$.

  Now suppose we have a $\lambda$-sequence $X\colon \lambda \to \mathcal{C}$ for some ordinal 
  $\lambda$ so that the morphisms $X_i\to X_{i+1}$  are pushouts of morphisms in $I$.  
This implies that $F$ preserves pushouts along all of the morphisms $X_i\to X_{i+1}$.
We will show that $F$ preserves pushouts along the transfinite composition $X_0\to \colim _{\lambda} X_{\beta}$. 
  
  We proceed by transfinite induction. Assume the claim is already proven for all ordinals smaller than $\lambda$.  
  Recall that for any limit ordinal $\beta<\lambda$, the induced map $\colim_{i<\beta} X_i \to X_{\beta}$ is an isomorphism 
  by the definition of a $\lambda$-sequence. Note furthermore that for a non-limit ordinal, say, $\beta+1$, 
  the indexing category $i<\beta+1$ has the terminal object $\beta$, so $\colim_{i<\beta+1} X_i \to X_{\beta}$ is an isomorphism. 

  Assume first  that $\lambda=\beta+1$ is not a limit ordinal. Then we have a diagram of pushouts   
  \[\xymatrix{
   X_0 \ar[r]\ar[d] & \colim_{i<\beta} X_i \ar[r]\ar[d]& \colim_{i<\beta+1} X_i\ar[d]\\
   C\ar[r] & D\ar[r] & D'.
   }\]
  Since the map $\colim_{i<\beta} X_i \to \colim_{i<\beta+1} X_i$ is either an isomorphism or the given map $X_{\beta-1}\to X_{\beta}$ 
  (where $\beta-1$ denotes the predecessor of $\beta$ in this case), the functor $F$ preserves the smaller pushouts squares 
  (for the left one, we use the induction hypothesis), and thus also the outer pushout rectangle.

  Now assume that $\lambda$ is a limit ordinal. Since colimits commute with each other we have 
  \[\colim(C\leftarrow X_0\rightarrow \colim_{\lambda} X)\cong  \colim_{\lambda}\colim(C\leftarrow X_0\rightarrow X_{\beta}).\]
  Note that the morphisms in the filtered colimit on the right are maps in $J$ since  they are pushouts of composites of maps in the diagram $X$.
  Using this observation and the assumption that $F$ commutes with filtered colimits along morphisms in $J$, we obtain
  \[F(\colim(C\leftarrow X_0\rightarrow \colim_{\lambda} X))\cong  \colim_{\lambda}F(\colim(C\leftarrow X_0\rightarrow X_{\beta})).\]
  The induction hypothesis allows us to replace the right hand side by 
  \[\colim_{\lambda}\colim(FC\leftarrow FX_0\rightarrow FX_{\beta})\] and we can exchange the colimits to replace the colimit above by
  \[\colim\colim_{\lambda}(FC\leftarrow FX_0\rightarrow FX_{\beta})\cong  \colim(FC\leftarrow FX_0\rightarrow \colim_{\lambda}FX_{\beta}).\]
  Finally we observe that $F$ preserves filtered colimits along morphisms in $J$  to see that  
  \[\colim(FC\leftarrow FX_0\rightarrow \colim_{\lambda}FX_{\beta})\cong \colim(FC\leftarrow FX_0\rightarrow F\colim_{\lambda}X_{\beta}).\]

  For the condition on retracts, suppose that $i'\colon B\to B'$ is a retract of a map $i\colon A\to A'$, and $F$ preserves pushouts along $i$.  Then we have a diagram 
  \[\xymatrix{
  B\ar[r]\ar[d]^{i'}&A\ar[r]\ar[d]^i&B\ar[d]^{i'}
  \\
  B'\ar[r]&A'\ar[r]&B'  
  }\]
  where both horizontal composites are the identity.  If we take the pushout of this diagram along a map $f\colon B\to C$ we obtain
  pushout squares 
  \[\xymatrix{B\ar[r]\ar[d]^{i'}&C\ar[d]&A\ar[r]\ar[d]^i&C\ar[d]
  \\
  B'\ar[r]&Q & A'\ar[r]&P
  }\]
  and the left hand square is a retract of the right hand square.  Applying $F$ to both squares preserves the retraction and the right pushout square.
  Since a retract of a pushout square is a pushout, $F$ applied to the left pushout square is a pushout.
\end{proof}

We also use this lemma in the proof of \autoref{prop:left}.

\begin{proof}[Proof of \autoref{prop:left}]
  First observe that since $\cC$ is left proper   $\oc$ is also left proper \cite[Theorem 13.1.14]{hirschhorn}.

  Since $\cC$ is cofibrantly generated both $\oc$ and $\gc$ are cofibrantly generated and generating cofibrations for both $\gc$ and $\oc$ can be 
  defined in terms of the generating cofibrations of $\cC$.  In fact, we can choose generating cofibrations $I$ for $\gc$ so that 
  $\Phi I$ is a collection of generating cofibrations for $\oc$ \cite{stephanthesis}.  
  Since $\Phi$ preserves retracts, filtered colimits, and pushouts along generating cofibrations, we see that $\Phi$ preserves cofibrations.

 By assumption, the fixed point functor $(-)^H$ preserves pushouts along generating cofibrations in $\gc$, so by \autoref{gencof}, it also preserves pushouts along all cofibrations.  It follows that the functor $\Phi$ also preserves pushouts along cofibrations.

  Consider a pushout diagram 
  \[\xymatrix{
   X \ar[r]^{h}\ar[d]_{f} & Y \ar[d]^{f'}\\
    Z \ar[r] & W
   }\]
  in $\gc$ where $f$ is a weak equivalence and $h$ is a cofibration. Applying $\Phi$ we have a pushout diagram in $\oc$ and by construction of the 
  model structure on $\gc$, $\Phi(f)$ is a weak equivalence.  By the observations above $\Phi(h)$ is a cofibration.  
  It follows that $\Phi(f')$ is a weak equivalence in $\oc$, so by definition $f'$ is a weak equivalence in $\gc$.
\end{proof}

\bibliographystyle{amsalpha}
\bibliography{ref}

\end{document}